\documentclass[11pt,leqno]{amsart}
\usepackage[mathcal]{eucal}
\usepackage[utf8]{inputenc}
\usepackage{amsfonts,amssymb}
\usepackage{hyperref}

\usepackage{tikz}
\tikzset{every picture/.style={line width=0.75pt}} 

\numberwithin{equation}{section}

\theoremstyle{plain}
\newtheorem{lemma}{Lemma}[section]
\newtheorem{proposition}[lemma]{Proposition}
\newtheorem{theorem}[lemma]{Theorem}
\newtheorem{corollary}[lemma]{Corollary}

\theoremstyle{remark}
\newtheorem{remark}[lemma]{Remark}
\newtheorem*{notation*}{Notation}

\newtheorem{questions*}{Questions}

\theoremstyle{definition}

\newtheorem{example}[lemma]{Example}

\DeclareMathOperator{\Stab}{Stab}
\DeclareMathOperator{\End}{End}
\DeclareMathOperator{\len}{len}
\DeclareMathOperator{\ima}{im}

\def\N{\mathbb{N}}
\def\Z{\mathbb{Z}}
\def\a{\textbf{a}}
\def\b{\textbf{b}}
\def\h{\textbf{h}}
\def\g{\textbf{g}}
\def\Fa{\mathcal{F}}
\def\Ca{\mathcal{C}}

\makeatletter
\def\moverlay{\mathpalette\mov@rlay}
\def\mov@rlay#1#2{\leavevmode\vtop{%
   \baselineskip\z@skip \lineskiplimit-\maxdimen
   \ialign{\hfil$\m@th#1##$\hfil\cr#2\crcr}}}
\newcommand{\charfusion}[3][\mathord]{
    #1{\ifx#1\mathop\vphantom{#2}\fi
        \mathpalette\mov@rlay{#2\cr#3}
      }
    \ifx#1\mathop\expandafter\displaylimits\fi}
\makeatother

\makeatletter
\def\author@andify{%
  \nxandlist {\unskip ,\penalty-1 \space\ignorespaces}%
    {\unskip {} \@@and~}%
    {\unskip \penalty-2 \space \@@and~}%
}
\makeatother
\title[Strong conciseness in profinite metabelian groups]{Strong conciseness in profinite metabelian groups}

\author[A.~Zozaya]{Andoni Zozaya} 
\address{Department of Statistics, Computer Science and Mathematics, Public University of Navarra, Campus of Arrosadia, 31006 Pamplona, Spain. Also affiliated with: INAMAT$^{2}$, Campus of Arrosadia, 31006 Pamplona, Spain}
\email{andoni.zozaya@unavarra.es}

\date{}

\makeatletter
\@namedef{subjclassname@2020}{\textup{2020} Mathematics Subject
  Classification}
\makeatother

\begin{document}

\begin{abstract}

A group word $w$ is said to be strongly concise in a class $\Ca$ of profinite groups if, for every group $G \in \Ca$ such that $w$ takes less than $2^{\aleph_0}$  values in $G$, the verbal subgroup $w(G)$ is finite. 

Using the notion of polynomial mappings introduced by Passi, we establish that every group word is strongly concise in the class of metabelian profinite groups. With this new approach, we also give an alternative proof for the fact that every group word is strongly concise in the class of profinite nilpotent groups, originally due to Detomi, Klopsch, and Shumyatsky. 
\end{abstract}

\maketitle


\section{Introduction}

A \emph{group word} $w$ in $k$ variables is an element of the free group $F_k$, and substituting variables for elements of a given group $G$, defines a map 
$w \colon G^k \rightarrow G$. Verbal problems study the relation between the image of the map, the so-called set of \emph{$w$-values}
\[ w\{G \}= \{ w(g_1, \dots, g_k) \mid g_i \in G \} \, , \]
and the \emph{verbal subgroup} $w(G) $ of $ G$ generated by \(w\{G\}\).  (In the context of topological groups, $w(G)$ denotes the \emph{closed} subgroup generated by $w\{G\}$.)

The word $w$ is said to be \emph{concise} in the class of groups $\Ca$, if for every $G$ in $\Ca$, $w(G)$ is finite whenever $w\{G\}$ is finite. For instance, a classical theorem of Schur \cite{Schur} says that the commutator word $\gamma_2(x, y) = [x, y]$ is concise in the class of all groups, and many other families of words are known to be concise. This conclusion fails for words in general: there are words and groups where $w\{G\}$ is finite, but, nevertheless, $w(G)$ is infinite, see \cite{Ivanov}. 

The property of conciseness has been generalised in several directions, we refer to \cite{Martina, semiconciseness, bounded, A8}. One such generalisation is that of \emph{strong conciseness} in the setting of profinite groups, introduced in \cite{DMS2} and subsequently extended in \cite{DKS, DMS, Iker,  A7, KS}. A word $w$ is strongly concise in a class of profinite groups $\Ca$, if for every $G$ in $\Ca$,  $w(G)$ is finite whenever $|w\{G\}|< 2^{\aleph_0}$. As a strengthened version of the classical conjecture that every word is concise in the class of profinite groups, see \cite{words}, it is conjectured that every word is actually strongly concise in the class of profinite groups, see \cite{DKS}.

The so-called \emph{strong conciseness conjecture} has been confirmed for the classes of virtually nilpotent profinite groups \cite{Detomi, DKS}, analytic groups over general pro-$p$ domains \cite{A2}, or linear profinite groups \cite{A7}. In most of these cases, one shows that $|w\{G\}|< 2^{\aleph_0}$ already forces $w\{G\}$ to be finite, so the result follows from the known conciseness of linear groups \cite{merz}, general analytic groups~\cite{A2}, or virtually abelian-by-nilpotent groups~\cite{TS}. Notwithstanding, there is yet another class of \textit{abstract} groups where every word is concise, but the strong version of the problem is still open: by a theorem of Turner-Smith \cite{TS}, every word is concise in an abelian-by-nilpotent group, but strong conciseness is unknown for their profinite counterparts, compare with \cite[Question~1.7]{A7}. 

\medskip

In this manuscript, we present the following case:

\begin{theorem}\label{thm: metabelian}
Every word is strongly concise in the class of profinite meta\-be\-lian
groups.
\end{theorem}

Although it is straightforward that every word is strongly concise in the class of abelian groups (see Lemma~\ref{lem: abelian}), strong conciseness does not behave well under group extensions, and the considerations that enter into its proof are new.  Indeed, the main idea is proving that when a word map between profinite spaces is polynomial in the sense of Passi (see Section \ref{sec: polynomial} for the precise definition), then its image is either finite or at least continuum sized, see Proposition~\ref{prop: polynomial strongly concise}.

This approach differs from the previous study of word problems in the class of abelian-by-nilpotent groups ({\it e.g.} \cite{Stroud} or \cite[Theorem 1.4]{A7}), which rely heavily on the fact that for a finitely generated nilpotent group $N$, the group ring $\Z[N]$ is Noetherian. This framework, however, is not available in general.

\medskip

It can be easily concluded from \cite{Leibman} that word maps in nilpotent groups are polynomial, which, as a by-product, gives an alternative proof for the strong conciseness of words in the class of nilpotent profinite groups, due to Detomi, Klopsch and Shumyatsky.

\begin{theorem}[\textup{\cite[Theorem 1.2]{DKS}}] \label{thm: nilpotent}
Every word is strongly concise in the class of profinite nilpotent groups. 
\end{theorem}

However, for metabelian groups, not all word maps are polynomial, as we show in Example~\ref{ex: metabelian pol}, but they are for commutator words whose image is central, see Theorem~\ref{thm: polynomial}. This is enough for proving Theorem~\ref{thm: metabelian}.

\subsection{Reader's guide} We start by proving, in Section~\ref{sec: split}, without relying on polynomial maps, that every word is strongly concise in profinite groups $G$ that are the product of two abelian groups; in particular, in split extensions of two abelian groups. In Section~\ref{sec: marginal}, we use the split case to prove a reduction to central-valued word maps. Section~\ref{sec: polynomial} is devoted to polynomial maps, and proving that a central-valued commutator word map in a metabelian group is indeed polynomial. We also reprove Theorem~\ref{thm: nilpotent} in the same section. Finally, we prove Theorem~\ref{thm: metabelian} in Section~\ref{sec: final}. 

\section{Products of two abelian groups} \label{sec: split}

In this section, we consider profinite groups $G$ that are a product $AB$ of two profinite abelian groups such that $A \unlhd G$. These are metabelian by a classical result of It\^{o} \cite{Ito}. (This factorisation is detected in the finite levels of the inverse system \cite{her}.) 

 \smallskip
 
Fix a $k$-variant word $w$ 
$$w(x_1,\dots,x_k) = \prod_{j=1}^\ell x_{i_j}^{\epsilon_j} \, ,$$
with $i_j \in \{1, \dots , k \}$ and $\epsilon_j \in \{\pm 1 \}$. The integer $\ell$ is called the \emph{length} of the word $w$, denoted by $\len(w)$.   We must show that $w\{G \}$ is finite, $|w\{G\}|<2^{\aleph_0}$, so that the verbal subgroup $w(G)$ is then also finite, by virtue of the Theorem of Turner-Smith \cite{TS}. That is precisely what is done in Theorem~\ref{thm: split}.

\smallskip

\subsection{Abelian profinite groups} 

On the one hand, every word is strongly concise in abelian profinite group.

\begin{lemma} \label{lem: abelian}
Every word is strongly concise in the class of profinite abelian groups.
\end{lemma}
\begin{proof}
When $G$ is abelian, $w\{G\}=w(G)$, and, since $w\{G\}$ is closed, then $w\{G\}$ is itself a profinite group. In particular, $w\{G\}$ is either finite or at least continuum sized. 
\end{proof}
Applied to $B$, the inequality
\[|w\{B\}|\le |w\{G\}|<2^{\aleph_0} \]
implies that $w(B)$ is finite.

\smallskip

\subsection{Generalised words} Once we fix a tuple $\b = (b_1, \dots, b_k) \in B^k$, we can define a so-called \emph{generalised word map} 
\begin{equation} \label{eq: generalised}
w_\b \colon A^k \rightarrow A, \qquad (a_1, \dots, a_k) \mapsto w(a_1 b_1, \dots, a_k b_k) w(\b)^{-1} \, .
\end{equation}
Observe that 
\[ w_{\mathbf b} \colon A^k \rightarrow A, \qquad
(a_1,\dots,a_k)\mapsto \prod_{j=1}^{\len(w)} a_{i_j}^{\, \nu_j(\mathbf b)} \, ,\]
where $\nu_j(x_1, \dots, x_k)$ are fixed words depending on $w$; compare with \cite[Equation 2.1]{words}. In particular, 
\begin{equation}
\label{eq: expression of words}
w\{G \} \subseteq \bigcup_{\b \in B^k} w_\b\{A\} \cdot w\{B\} \, 
\end{equation}
where $w_\b\{ A \}$ is the image of the generalised word map $w_\b$.

Since $A$ is abelian, each $w_{\b}$ is a continuous homomorphism of abelian profinite groups. In particular, the generalised verbal subgroup $w_{\b}(A) = \ima{w_\b}$ is a closed subgroup of $A$, {\it i.e.} a profinite group. Moreover,
\[ |w_\b(A) | \leq |w\{G\}|^2 < 2^{\aleph_0}\, .\]
Since every profinite group is either finite or has order at least the continuum, then $\ima(w_\b)$ is finite. However, it might happen that the union 
\[ \bigcup_{\b \in B^k} w_{\b}(A)\]
is infinite, even though each $w_\b(A)$ is finite.

\medskip

We may also fix a tuple $\a \in A^k$ and will study the image of the map 
\[ w_\a \colon B^k \rightarrow A, \quad \b \mapsto w(\a \b) w(\b)^{-1} \, . \]

Note that $w_{\a}$ is not a generalised word. But using the additive notation, from \eqref{eq: expression of words}, we can recover the expression
\begin{equation}
\label{eq: wa}
w_{\a}(\b) = \sum_{j=1}^{\len(w)} a_{i_j} \cdot \nu_j(\b) \, .
\end{equation}

Hence, set $\ell = \len(w)$,  we have the continuous map
$$\Lambda \colon B^k \rightarrow B^\ell, \, \quad \b \mapsto (\nu_1(\b), \dots, \nu_\ell(\b)),$$
whose image $\Lambda(B^k)$ is a closed subgroup of $B^\ell$. Indeed, component-wise $\Lambda(B^k)$ is a verbal subgroup of the abelian group $B$, and so it follows from Lemma~\ref{lem: abelian}.

\medskip

We collect this reasoning in the following lemma:

\begin{lemma}
\label{lem: image}
Let $\ell = \len(w)$. There exists a closed subgroup $H \leq B^\ell$ such that
\[ w_{\a}(B^k) = \left\{ \sum_{j=1}^\ell a_{i_j} \cdot h_j  \mid (h_1, \dots, h_\ell) \in H \right\}.\]
\end{lemma}

Next we study the possible cardinalities of these sets. For a profinite group $B$ acting continuously on an abelian profinite group $A$, the \emph{stabiliser} of $S \subseteq A$ in $H \leq B$, is
\[ \operatorname{Stab}_H(S) = \{ h \in H \mid a \cdot h = a, \quad \forall a \in S \}.\]
\begin{lemma}
Let $H$ be an abelian profinite group acting continuously on an abelian profinite group $A$, and let $a \in A$. The stabiliser $\operatorname{Stab}_H(a)$ is a closed normal subgroup of $H$. 
\end{lemma}
\begin{proof}
Since the action is continuous, the orbit map $ \theta \colon  H \rightarrow A$, $h \mapsto a \cdot h$ is continuous, and hence $\Stab_H(a) = \theta^{-1}(a)$ is a closed subset of $H$. 
\smallskip

It remains to prove that $\Stab_H(a)$ is a subgroup. Let $h_1, \, h_2 \in \Stab_H(a)$. Then,
\[ a \cdot (h_1h_2)=(a\cdot h_1)\cdot h_2=a\cdot h_2=a, \]
so $h_1h_2\in \Stab_H(a)$. Also, if $h\in \Stab_H(a)$, then
\[ a= a\cdot 1=a\cdot (hh^{-1})=(a\cdot h)\cdot h^{-1}=a \cdot h^{-1}, \]
so $h^{-1}\in \Stab_H(a)$. Thus $\Stab_H(a)$ is a subgroup of $H$. Finally, since $H$ is abelian, all its subgroups are normal.
\end{proof}

We will also use the following result.

\begin{proposition}[\textup{\cite[Proposition 2.1]{DKS}}] \label{prop: DKS}
 Let $\phi \colon X \rightarrow Y$  be a continuous map between non-empty Stone spaces. Suppose that $|\phi(X)| < 2^{\aleph_0}$. Then, there exists a non-empty open $U \subseteq X$ where $\phi$ is constant.
\end{proposition}

\begin{proposition}
Let $A$ be an abelian profinite group, let $B$ be an abelian profinite group acting continuously on $A$, and let $H \leq B^m$ be a closed subgroup. For a fixed tuple $\mathbf a = (a_1, \dots, a_m) \in A^m$, define
\[
F_{\mathbf a}\colon H \to A,\qquad (h_1,\dots,h_m)\mapsto \sum_{j=1}^m a_j\cdot h_j.
\]
Then either $F_{\mathbf a}(H)$ is finite, or $|F_{\mathbf a}(H)|\ge 2^{\aleph_0}$.
\end{proposition}

\begin{proof}
We proceed by induction on $m$. When $m=1$,  $F_{\mathbf a}(H)=a_1 \cdot H$ is the orbit of $a_1$, and by the Orbit Stabiliser Lemma,
\[ |a_1\cdot H| = \bigl|H/\Stab_H(a_1)\bigr|. \]
Since $\Stab_H(a_1)$ is a closed normal subgroup of the profinite group $H$, the quotient $H/\Stab_H(a_1)$ is a profinite group; and, hence, it is either finite or has cardinality at least the continuum.

\medskip

Assume by induction hypothesis the statement for $m-1$, and let $H\leq B^m$ be closed. Assume that $|F_\a(H)| < 2^{\aleph_0}$. Since $F_\a$ is a continuous map between Stone spaces, there exists an open coset $g K \subseteq B^m$ such that $F_\a$ is constant when restricted to $gK$. Choose a finite set $R \subseteq H$ such that $\{ gh \mid h \in R \}$  is a left-transversal for $K$ in $H$, then
\[H=\bigcup_{h \in R} g h K,
\quad \text{ and } \quad F_{\mathbf a}(H)=\bigcup_{ h \in R} F_{\mathbf a}(  g h K) \, . \]

Furthermore, if $h=(h_1,\dots, h_m)\in R$ and $k=(k_1,\dots, k_m)\in K$, with $h_j , \, k_j \in B$, then
\[
F_{\mathbf a}(g h k)
= \sum_{j=1}^m a_j\cdot (g_j h_j k_j)
= \sum_{j=1}^m (a_j \cdot g_j)\cdot h_j  k_j
= F_{\mathbf a \cdot \textbf{g}}(h k),
\]
where $\mathbf a\cdot \textbf{g} =(a_1\cdot g_1,\dots,a_m\cdot g_m) \in A^m$.
Hence, we may assume without lose generality that $g = \textbf{1}$, and hence 
\[ F_{\a}(K) = F_\a (1, \dots, 1) = \sum_{j=1}^m a_j\]
is constant, {\it i.e.}
\begin{equation}
\label{eq: 1}
\sum_{j=1}^m a_j \cdot k_j - a_j = 0 \, ,
\end{equation}
for all $(k_1, \dots, k_m) \in K$.

\smallskip

Choose $h \in R$ and $k \in K$, then 
\[ F_\a(hk) - F_\a(h) = \sum_{j=1}^m a_j \cdot (h_j k_j ) - a_j \cdot h_j = \sum_{j=1}^{m} (a_j \cdot k_j - a_j) \cdot h_j \, . \]

Substituting \eqref{eq: 1} and using that $B$ is abelian, we get
\begin{align*}
F_\a(hk)-F_\a(h)
&= \sum_{j=1}^{m-1} (a_j \cdot k_j-a_j)\cdot h_j + (a_m  \cdot k_m-a_m)\cdot h_m \\
&= \sum_{j=1}^{m-1} (a_j \cdot k_j-a_j)\cdot h_j
   - \sum_{j=1}^{m-1} (a_j \cdot k_j-a_j)\cdot h_m \\
&= \sum_{j=1}^{m-1}
   \Bigl( (a_j \cdot h_j-a_j \cdot h_m)\cdot k_j - (a_j \cdot h_j-a_j\cdot h_m ) \Bigr) \, .
\end{align*}

Therefore, define, for each $j=1, \dots, m-1$, the fixed value $b_j(h) = a_j \cdot h_j - a_j \cdot h_m$. Then 
\begin{equation}
\label{eq: 2} 
F_\a(hk)= F_\a(h) - \sum_{j=1}^{m-1} b_j(h) + \sum_{j=1}^{m-1} b_j(h) \cdot k_j \, .
\end{equation}

Let $\pi \colon K \rightarrow B^{m-1}$ be the projection to the first $m-1$ coordinates of $B^m$, so that $\check{K} = \pi(K) \leq B^{m-1}$ is a closed subgroup. Furthermore, since translations preserve cardinality, and, since the formula \eqref{eq: 2} depends only on the first $m-1$ coordinates of $k \in K$, we have 
\[ \# F_\a (hK)  = \# \left\{ \sum_{j=1}^{m-1} b_j(h) \cdot k_j \mid (k_1, \dots, k_{m-1}) \in \check K \right\} \, , \]
and the left-hand side is either finite or has at least $2^{\aleph_0}$ points, by the induction hypothesis. 

\smallskip

Finally, either there exists $h_0 \in R$ such that $|F(h_0K)| \geq 2^{\aleph_0}$, or otherwise, since $K$ has finite index in $H$, then 
\[ F_\a (H) = \bigcup_{h \in R} F_\a( h K) \]
is finite, as it is a finite union of finite sets. This proves the statement.
\end{proof}

From the proposition above and Lemma~\ref{lem: image}, we conclude
\begin{corollary}
\label{cor: wa}
Let $\a \in A^k$ and let $w_\a$ be defined as in \eqref{eq: wa}. If $|w\{G\}|< 2^{\aleph_0}$, then, $\ima{w_\a}$ is finite.
\end{corollary}

\subsection{Finitely many homomorphisms}

For every tuple $\a \in A^k$ or $\b \in B^k$, we know that the images of the maps $w_\a$ and  the homomorphisms $w_\b$ are finite, provided that $|w\{G\}| < 2^{\aleph_0}$. For this purpose we consider the \emph{compact-open} topology in $\Ca(X, Y)$ -- the continuous maps between two topological spaces $X$ and $Y$ -- to prove that the finiteness of all these images forces that there are finitely many of these maps. The neighbourhood basis of the above-mentioned topology in $\Ca(X, Y)$ is given by 
\[ [K, U] = \left\{ f \in \Ca(X, Y) \mid f(K) \subseteq U \right\},\]
where $K$ is a compact subset of $X$, and $U$ is an open subset of $Y$.

\begin{lemma}
\label{lem: compact-open}
Let $A$ and $C$ be profinite abelian groups, and let $\Fa$ be a compact subset of $\Ca(A, C)$ endowed with the compact-open topology. Suppose that $\Fa$ consists of group homomorphisms, and that for each $a \in A$, the set 
$$ \{ f(a) \mid f \in \Fa \} $$
is finite. Then, $\Fa$ is finite. 
\end{lemma}

\begin{proof}
Suppose first that $C$ is countably based, and hence metrizable. The compact-open topology on $\Ca(A, \, C)$ coincides with the topology of uniform convergence (see \cite[\S \, 46, Theorem 46.8]{Munkres}). In particular, $\Ca(A, \, C)$ is metrizable, and $\Fa$ is sequentially compact.

\smallskip
Suppose, by contradiction, that $\Fa$ is infinite, so that there exists an accumulation point $f \in \Fa$, {\it i.e.}, there exists a sequence of pairwise distinct maps $(f_n)_{n \in \N}$ in $\Fa$ that converges uniformly to $f$.

In particular, for each $a \in A$, the sequence $(f_n(a))$ in $A$ converges to $f(a)$ in $C$. Since by hypothesis $\{ f(a) \mid f \in \Fa \}$ is finite, the sequence $(f_n(a))$ stabilises, {\it i.e.} for each $a \in A$, there exists $n_a \in \N$ such that 
\[ f_n(a) = f(a), \, \text{ for all } n \geq n_a. \]

\smallskip

For each integer $n \in \N$, define the subgroup
\[ K_n = \bigcap_{m \geq n} \ker(f_m - f) . \]
Each $K_n$ is closed in $A$, and we have already proved that $ A = \bigcup_{n \in \N} K_n$. 

\smallskip
Hence, by the Baire Category Theorem, there exists $\ell \in \N$ such that $K_\ell$ has non-empty interior, and thus $K_\ell$ is a normal open subgroup in $A$. Moreover, if $n \geq \ell$,
\[  f_n(x) = f(x), \,  \text{ for all } \, x \in K_\ell \, . \]

\smallskip

Since $A/K_\ell$ is finite, choose representatives $a_1, \dots, a_k$ for the cosets of $K_\ell$ in $A$. Recall that for each $i$ the sequence $(f_n(a_i))$ stabilises, so there exists $\ell_i \in \N$ such that
\[ f_n(a_i) = f(a_i), \quad \text{for all } n \geq \ell_i. \]

Let $N = \max\{\ell, \ell_1, \dots, \ell_k\}$ and $n \geq N$. Take $a \in A$, and write $a = a_i + x$ with $x \in K_\ell$. Then,
\[ f_n(a) = f_n(a_i) + f_n(x) = f(a_i) + f(x) = f(a) \, . \]
Thus $f_n = f$ for all $n \geq N$, contradicting the assumption that the maps $f_n$ are pairwise distinct. Therefore, $\Fa$ must be finite.

\smallskip 

We now reduce the general case to the countably based one. For that purpose, we consider quotient maps $q_N \colon C \rightarrow C/N$ with $N \unlhd C$ a closed subgroup. In particular, when $N$ is open, $C/N$ is finite and $\Ca(A, \, C/N)$ is discrete. Thus each
\[ \Fa_N = \{ q_N \circ f \mid f \in \Fa \} \subseteq \Ca(A, \, C/N )\]
is finite, because it is a compact subset of a discrete space. 

Assume by contradiction that $\mathcal F$ is infinite, then we can build a descending sequence of open subgroups 
\[ C=N_0\geq N_1\geq N_2\geq \cdots \]
so that the sets $\mathcal F_{N_i}$ have strictly increasing
cardinalities. 

Let $N= \cap_{i \in \N} N_i$. On the one hand, $\Fa_N$ in $\Ca(A, C/N)$ is infinite, as the reductions modulo $N_i$ have increasing cardinalities. On the other hand, the open subgroups $N_i/N$ form a neighbourhood basis for $C/N$, so that $C/N$ is a countably based profinite group. Thus, since $\Fa_N$ is compact in $\Ca(A, C/N)$, it is finite. This contradiction yields that $\Fa$ is finite.
\end{proof}

This allows us to prove the main result of the section:

\begin{theorem}\label{thm: split}
Let $w$ be a group word and let $G$ be a profinite group. Suppose that $G =AB$ where $A \unlhd G$ and $B \leq G$ are abelian profinite subgroups of $G$, and that $|w\{G\}| < 2^{\aleph_0}$. Then, $w\{G\}$ is finite.
\end{theorem}

\begin{proof}
On the one hand, applying Lemma~\ref{lem: abelian} to $B$, since $|w\{B\}| \leq |w\{G\}| < 2^{\aleph_0}$, then $w(B)$ is finite. On the other hand, for each fixed $\b \in B^k$, $\ima{w_\b}$ is also finite. We shall prove that the set of homomorphisms 
\[ \Fa = \{w_{\mathbf b}\mid \mathbf b\in B^k\}\subseteq \Ca(A^k, \, A)
\]
is finite. Then, $w\{G \}$ is finite in view of the formula \eqref{eq: expression of words}.  

\medskip 

In order to apply Lemma~\ref{lem: compact-open}, we shall prove that $\Fa$ is compact with the compact-open topology in $\Ca(A^k, \, A)$, and that for each $\a \in  A^k$ the set 
\[ \{ w_{\b}(\a) \mid \b \in B^k \}\]
is finite. The latter follows by Corollary~\ref{cor: wa}, since
\[ \{ w_{\b}(\a) \mid \b \in B^k  \} = \ima{ w_{\a}} \subseteq w\{G\}\cdot w\{G\}^{-1} \, ,\]
and hence $|\ima{w_{\a}}| < 2^{\aleph_0}$.

For the former, observe that the map 
\[ \Phi \colon B^k \rightarrow \Ca(A^k, A) \quad \b \mapsto w_\b\]
is continuous, where $\Ca(A^k, A)$ is endowed with the compact-open topology.  Indeed, each evaluation map
\[ B^k\times A^k \to A, \qquad (\mathbf b,\mathbf a) \mapsto w_{\mathbf b}(\mathbf a) \]
is continuous. Hence, $\Fa = \Phi(B^k)$ is compact, as it is the image of a compact set by a continuous map. 
\end{proof}

\section{Marginality and centrality} \label{sec: marginal}

In this section, we prove two reductions that will be useful to prove Theorem~\ref{thm: metabelian}. The first is about the \emph{marginal subgroup} of a word $w$ in $G$, namely  
\[
w^*(G)=
\left\{\, x\in G \ \middle|\
\begin{array}{l}
w(g_1,\ldots, g_ix,\ldots,g_k)
= w(g_1,\ldots,g_i,\ldots,g_k), \\[2mm]
\forall\, i=1,\ldots,k,\ \forall\, g_1,\ldots,g_k\in G
\end{array}
\right\} \, .
\]

We say that $N \leq G$ is marginal for $w$ in $G$ if $N \leq w^*(G)$. 

\begin{proposition}\label{lem: marginal}
Let $G$ be a profinite metabelian group and let $w$ be a word. Suppose that $|w\{G \}| < 2^{\aleph_0}$. Then $G'$ is finite-by-marginal for $w$.
\end{proposition}
\begin{proof}
We denote $A=G'$ and $B=G/G'$. We work on the split metabelian group $S = A \rtimes B$. For $\a=(a_1, \dots, a_k) \in A^k$ and $\b = (b_1, \dots, b_k) \in B^k$, let $\h = (h_1, \dots, h_k) \in G^k$  be any lift of $\b$, so that 
\[ w(a_1 b_1, \dots, a_k b_k) = w_\h(a_1, \dots, a_k) w(\b) \, ,\]
where $w_\h \colon A^k \rightarrow A$ is the usual generalised map $\a \mapsto w(\a\h)w(\h)^{-1}$ in \eqref{eq: generalised} (in other words, the action of $G$ in $G'$ factors through the abelianization $G/G'$). In particular, 
\[ |w\{ S\} | \leq |w\{G \}|^2| w\{G/G'\}| < 2^{\aleph_0} \, . \]
Hence, by Theorem~\ref{thm: split}, $w(S)$ is finite. Let $N= A \cap w(S) \leq A$. Since $w(S)$ is normal in $S$, and the action of $G$ factors through $G/G'$, then $N \unlhd G$. Then, 
\[ w(a_1 h_1, \dots, a_k h_k) \equiv w(h_1, \dots, h_k) \pmod{N}, \quad \text{ for all } \, a_j \in G'  \, \text{ and } h_j \in G \, , \]
and thus, 
\[ w(\bar{a}_1 \bar{h}_1, \dots, \bar{a}_k \bar{h}_k)= w(\bar{h}_1, \dots, \bar{h}_k) \quad \text{ for all } \, \bar{a}_j \in G' /N \, \text{ and } \bar{h}_j \in G/N \, . \]
That is, $G'/N$ is marginal for $w$ in $G/N$. 
\end{proof}

\begin{proposition}\label{prop: semiconcise}
Let $G$ be a profinite metabelian group and let $w$ be a word. Suppose that $|w\{G \}| < 2^{\aleph_0}$. Then $[w(G), G]$ is finite. 
\end{proposition}
\begin{proof}
In view of Proposition~\ref{lem: marginal}, there exists a finite normal subgroup $N \unlhd G$ such that $N \leq G'$ and $G'/N$ is marginal for $w$ in $G/N$. Thus, if $g_1, \dots , g_k , t \in G$, then 
\[ w(g_1, \dots, g_k)^t \equiv w(g_1^t, \dots, g_k^t) \equiv w(g_1[g_1, t], \dots, g_k[g_k, t]) \equiv w(g_1, \dots, g_k) \mod{N} \, . \]
 Hence, $[w\{G \}, G] \leq N$, and thus $[w(G), G] \leq N$ is finite.
\end{proof}

\begin{remark}
We could have obtained the previous reduction more easily for pro-$p$ metabelian groups. Finitely generated abelian-by-nilpotent groups are verbally elliptic \cite{Stroud} (for every word $w$ there exists an integer $m \in \N$ such that every element of $w(G)$ is the product of at most $m$ $w$-values or inverses of $w$-values). Therefore, when $G$ is a pro-$p$ group, we could alternatively prove strong conciseness for all groups in the class of pro-$p$ metabelian groups by first reducing to the case where $w(G)$ is central by directly citing \cite[Corollary 4.2]{DKS}.
\end{remark}

\section{Polynomial maps}
\label{sec: polynomial}

We start by presenting the notion of \emph{polynomial map} between groups, introduced by Passi \cite{Passi}, we refer to \cite{Tao} for a more modern perspective. Intuitively, these are maps whose sufficiently large ``derivative'' vanishes.  More precisely, let $G$ and $A$ be groups, and $f \colon G \rightarrow A$. Define, for $h \in G$, a \emph{difference} operator as
$$\Delta_h f(x)= f(x)^{-1} f(xh) \, .$$
Then, $f$ is \emph{polynomial} of degree at most $n$ if all $(n+1)$ iterated differences vanish, {\it i.e.} 
\[ \Delta_{h_{n+1}} \dots \Delta_{h_1} f(x) = 1\, \]
 for any \(x,h_1,\dots,h_{n+1}\in G\).

 \smallskip
 
 When $A$ is abelian there is an equivalent characterisation in terms of the augmentation ideal. In this case, $f \colon G \rightarrow A$ can be linearly extended to a $\Z$-linear map $ \hat{f} \colon \Z[G] \rightarrow A$. Then $f$ is a \emph{polynomial map} of degree at most $n$ if $\hat{f} (I(G)^{n+1})=0$, where $I(G)$ is the augmentation ideal, namely the kernel of the augmentation map $$\varepsilon \colon \Z[G] \to \Z,  \quad \sum_{g \in G} a_g g \mapsto \sum_{g \in G} a_g \, . $$  

For instance, constant maps are polynomial of degree $0$ and nontrivial homomorphisms are polynomial of degree $1$. We simply say that $f$ is polynomial if it is polynomial of some finite degree.

\subsection{Polynomial maps in profinite groups}

We study the size of the image of a polynomial map.

\begin{proposition} \label{prop: polynomial strongly concise}
Let $G$ and $A$ be profinite groups, and let $f \colon G \rightarrow A$ be a continuous polynomial map. Then either $f(G)$ is finite or $|f(G)| \geq 2^{\aleph_0}$. 
\end{proposition}

\begin{proof}
We proceed by induction on the degree of $f$. If $f$ has degree $0$, then $f$ is constant. 

Let us denote by $d$ the degree of $f$, and assume, by induction hypothesis, that the statement holds for polynomial maps of degree $\leq d-1$. Assume, moreover, that $|f(G)| < 2^{\aleph_0}$. By Proposition~\ref{prop: DKS}, there exists an open subgroup $U \leq G$ such that $f$ is constant on a coset $xU$, {\it i.e.} $f(xu) = f(x)$, for all $u \in U$. In other words, if, for $g \in G$ we define the map $f_g \colon U \rightarrow A$ such that $f_g(u)=f(gu) \in A$, then $f_x$ is constant. 

We will prove that $f_g$ is polynomial of degree at most $d-1$ for every $g \in G$. Indeed, since $f$ is polynomial of degree $d$, any $d$-th difference
\[ \Delta_{u_1} \dots \Delta_{u_{d}} f(g) \]
is independent of $g \in G$ (its value is a constant depending only on $u_1, \dots, u_d \in U$). Therefore, 
\[ \Delta_{u_1} \dots \Delta_{u_{d}} f_g(u) = \Delta_{u_1} \dots \Delta_{u_{d}} f(gu) = \Delta_{u_1} \dots \Delta_{u_{d}} f(x) = \Delta_{u_1} \dots \Delta_{u_{d}} f_x(1) = 1 \, ,\]
for all $u \in U$, using for the last equality that $f_x$ is constant, and thus polynomial of degree $0$. 

Since $|f_g(U)| \leq |f(G)| < 2^{\aleph_0}$, then $f_g(U)$ is finite by the induction hypothesis. Finally, choose a finite transversal $T$ for $U$ in $G$, so that
\[ f(G)= \bigcup_{t \in T} f_t(U)\]
is also finite.
\end{proof}

For instance, word maps are polynomial in nilpotent groups:

\begin{proof}[Proof of Theorem~\ref{thm: nilpotent}]
If $f_1,  f_2 \colon G \rightarrow N$ are polynomial maps into a nilpotent group $N$, Leibman (see \cite[Theorem 3.2]{Leibman}) proves that their pointwise product $f_1 f_2 \colon G \rightarrow N$ where $(f_1 f_2)(x)=f_1(x)f_2(x)$, and the pointwise inverse $f_1^{-1} \colon G\rightarrow N$ where $g \mapsto f_1(g)^{-1}$ are also polynomial. Since the maps sending $(x_1, \dots, x_k) \in G^k$ to $x_i$ in $G$ are trivially polynomial of degree $1$, then it follows easily that, for a nilpotent group $N$, any word map $w \colon N^k \rightarrow N$ is polynomial. 
\smallskip

In particular, if $|w\{G\}| < 2^{\aleph_0}$, then $w\{G\}$ is finite, and thus so is $w(G)$ by Turner-Smith ({\it loc. cit.}).
\end{proof}

\subsection{Central valued word maps}

Henceforward, we will restrict ourselves to commutator words, that is to words $w \in F_k'$. This is enough for our purposes because of the following result:

\begin{proposition}[\textup{\cite[Lemma~6.1]{DKS}}]\label{prop: commutator}
Let $\Ca$ be a class of profinite groups where every commutator word is strongly concise. Then every word is strongly concise in $\Ca$.
\end{proposition}

In view of Proposition~\ref{prop: polynomial strongly concise}, it is natural to try proving that word maps on metabelian groups are polynomial. However, the statement is false in this generality, as the following example witnesses:

\begin{example} \label{ex: metabelian pol}
Inspired by \cite[Example~1]{Hu}, we consider the metabelian group
\[ G = \langle x, y ,z \mid [x, y] =1, \, x^z = y,  \, y^z = xy \rangle \, , \]
so that $G' = \langle x, y \rangle \cong \Z^2$. If the commutator map $\gamma_2 \colon G^2 \rightarrow G$ were polynomial, so would be the composition
\[ f \colon \Z \rightarrow G^2 \rightarrow G' \, , \qquad n \mapsto (y , z^n) \mapsto \left[y, z^n \right] \, . \]
We can show by induction that 
\[ f(n) = y^{-1} z^{-n} y z^n  = x^{F_{n}} y^{F_{n+1}-1} , \]
where $F_n$ is the $n$-th term of the Fibonacci sequence with $F_0 =0$ and $F_1 =1$. Hence, if $f$ were polynomial, so would the mapping $\tilde{f} \colon \Z \rightarrow \Z^2$ with $n \mapsto (F_n, \, F_{n+1})$. However, 
\[ \Delta_1 \tilde{f}(n)= \tilde{f}(n-1) ,\]
for any $n \geq 2$, and thus, $f$ is not polynomial.
\end{example}

Nevertheless, we aim to prove the following result:

\begin{theorem} \label{thm: polynomial}
Let $w$ be a commutator word and $G$ a metabelian group. Suppose that $w\{G \} \subseteq Z(G)$. Then, the word map $w \colon G^k \rightarrow Z(G)$ is polynomial.
\end{theorem}

Furthermore, since we are restricting to metabelian groups, we can regard group words as elements of the free metabelian group $M_k= F_k/F_k''$, or simply $M$ when $k$ is clear from the context, with basis $\{ x_1, \dots, x_k \}$. Moreover, $M'$ can be regarded as an $\Z[M/M']$-module under conjugation, and since $M/M' \cong \Z^k$ we identify this group algebra with the ring of Laurent polynomials
\[ \Z[M/M'] \cong \Z[\Z^k] \cong \Z[t_1^{\pm 1} , \dots , t_k^{\pm 1} ] \, .\]

With this identification, by a result of Bachmuth \cite{Ba}, any $w \in M'$ can be written in the form 
\[ w = \sum_{1 \leq j < i \leq k}  \lambda_{ij} [x_j, x_i] \, , \]
with $\lambda_{ij} \in \Z[M/M']$.

\smallskip

We are interested in the cancellation of the variables appearing in $w$. For that purpose, for each subset $T \subseteq \{ 1 , \dots, k \}$, we define the contraction $\rho_T \colon F_k \rightarrow F_k$ 
\[\rho_T(x_i)=
\begin{cases}
x_i,& i\in T,\\
1,& i \notin T \, .
\end{cases}
\]

We say that $w \in M'$ is \emph{reduced} in the $i$th variable when $\rho_{\{1, \dots, k\} \setminus \{i \}}(w) =1$, that is, 
\[ w(x_1, \dots, x_{i-1}, 1, x_{i+1}, \dots, x_k) = 1 \,  \]
in the free metabelian group.

\begin{lemma} \label{lem: reduced char}
Let $w \in M'$ be a commutator word. Then, $w$ is reduced in the variable $i$ if, and only if, there exist coefficients $\lambda_{j} \in \Z\left[t_1^{\pm 1}, \dots, t_k^{\pm 1} \right]$ such that 
\[ w =  \sum_{i \neq j} \lambda_{j} [x_i, x_j]\, .\]
\end{lemma}
\begin{proof}
It is clear that the right-hand side vanishes when $x_i=1$. For the other implication, let $T_i= \{ 1, \dots, k \} \setminus \{i \}$ and $w = \sum_{p, q} \lambda_{pq} [x_p, x_q] \in M'$ such that $\rho_{T_i}(w)=0$. Note that the summand $\lambda_{pq}[x_p, x_q]$ only vanishes when $p=i$, or $q=i$, or when $t_i-1$ divides $\lambda_{pq}$. 

Consider the evaluation homomorphism
$$\varepsilon_i \colon  \Z\left[t_1^{\pm 1}, \dots, t_k^{\pm 1} \right] \rightarrow  \Z\left[t_1^{\pm 1}, \, t_{i-1}^{\pm 1}, \, t_{i+1}^{\pm 1}, \dots, \, t_k^{\pm 1} \right]  $$ 
defined by $t_i \mapsto 1$ and  $t_j \mapsto t_j$ for  $j \neq i$,
so that $\ker \varepsilon_i$ is the ideal generated by $1-t_i$. We write
\[ \lambda_{pq} = \varepsilon_i(\lambda_{pq}) + (1- t_i) \mu_{pq} \, ,\]
for some $\mu_{pq} \in \Z\left[t_1^{\pm 1}, \dots, t_k^{\pm 1} \right]$. Then, 
\[ w = \sum_{i \neq p, q} \varepsilon_i(\lambda_{pq}) [x_p, x_q] + \sum_{i \neq p, q} (1-t_i) \mu_{pq} [x_p, x_q] + \sum_{i \in \{p, q\}} \lambda_{pq} [x_p, x_q] \, .\]
On the one hand, 
\[\sum_{i \neq p, q} \varepsilon_i(\lambda_{pq}) [x_p, x_q] = \rho_{T_i}(w) =0 \, .\]
On the other hand, since the Hall-Witt identity in the free metabelian group reads as 
\[ (1-t_i)[x_p, x_q]=(1-t_q)[x_p,x_i]-(1-t_p)[x_q, x_i] \, , \]
then $(1-t_i)\mu_{pq}[x_p, x_q] = \sum_{s =1}^k \nu_{s} [x_i, x_s]$ for suitable coefficients $\nu_{s}$. By adjusting the coefficients in $\Z\left[t_1^{\pm 1}, \dots, t_k^{\pm 1} \right]$ we obtain the desired expression. 
\end{proof}

\begin{proposition} \label{prop: polynomial-reduced}
Let $w$ be a commutator word and $G$ a metabelian group. Suppose that $w$ is reduced in all variables, and that $w\{G \} \subseteq Z(G)$. Then, the word map $w \colon G^k \rightarrow Z(G)$ is polynomial.
\end{proposition}
\begin{proof}
We start by setting some context. We write $G'$ additively as a $G/G'$-module, and we observe that all conjugation maps 
\[ c_g \colon G' \rightarrow G' \quad x \mapsto g^{-1}x g\]
are commuting automorphisms of $G'$. Moreover, in view of Lemma \ref{lem: reduced char}, since $w$ is reduced, we can assume without loss of generality that  
\[ w = \sum_{j \neq i} \lambda_{j} [x_j, x_i] \, ,\]
for a fixed variable $x_i$, and with $\lambda_{j} \in \Z[t_{1}^{\pm 1}, \dots, t_{k}^{\pm 1}]$. Moreover, we write each coefficient as a Laurent polynomial in the variable $t_i$, {\it i.e.}
\[
\lambda_{j}
=
\sum_{m\in F_i}\lambda_{j}^m t_i^m,
\]
where \(F_i\subseteq \Z\) is finite and $\lambda_{j}^m$ does not involve the variable $t_i$.  We denote by $\Psi_{j}^m$ the $\Z$-linear endomorphism of $G'$, corresponding to $\lambda_{j}^m$. Since $\Psi_{j}^m$ is the finite integral linear combination of suitable conjugation maps, it commutes with any $c_g$.

We fix all the variables except the $i$-th one. More precisely, choose arbitrary elements $a_1, \dots,  a_{i-1}, a_{i+1}  \dots, a_{k} \in G$, and define the generalised word map 
\[ W_{a} \colon G \rightarrow Z(G) \quad g \mapsto w(a_1, \dots, a_{i-1}, g, a_{i+1}, \dots, a_k) \, .\]
For notational convenience, define also $\delta_j \colon G \rightarrow G'$ as $\delta_j(g) = [a_j, g]$, so that 
\[
W_a
= \sum_{m\in F_i}\sum_{j\neq i} \Psi_{j}^{m} \cdot c_g^m \cdot \delta_j.
\]

We first study the differences of each summand separately. For each integer $r \in \Z$ and $h \in G$, define the twisted difference operator 
$$\nabla_{r, h} \colon \mathcal{C}(G, G') \rightarrow \mathcal{C}(G, G') \qquad \nabla_{r, h} F(g) = F(gh) - c_h^r \circ F(g) \, .$$

Observe that when $F$ is central-valued
\begin{equation}\label{eq: twisted central}
\nabla_{r, h} F =\Delta_{h} F \,
\end{equation}
for any $r \in \mathbb{Z}$. We study the effect of $\nabla_{r, h }$ in a summand of the form $\Psi c_g^m \delta_j$, for $\Psi \in \End_\Z(G')$ that commutes with $c_g$. Using the commutator identity $[x, yz]=[x,z][x,y]^z$, we get
\[
\begin{aligned}
\nabla_{r,h}\Psi c_g^m \delta_j(g)
&=
\Psi \, c_{gh}^m \delta_j(gh)-\Psi  \, c_g^m \, c_h^r \delta_j(g)\\
&=
\Psi \,c_g^m \, c_h^m \left(\delta_j(h)+c_h\delta_j(g) \right)
-
\Psi\,c_g^mc_h^r\delta_j(g)\\
&= \Psi\,c_g^m\, (c_h^{m+1}-c_h^r)\delta_j(g) + \Psi \,c_g^{m}c_h^m\delta_j(h) \, .
\end{aligned}
\]

Thus, $\nabla_{m+1, \,  h} \Psi c_g^m \delta_j(g)=  \Psi \,c_g^{m}c_h^m\delta_j(h)$.

\smallskip

For notational convenience, let us say that a map $G \rightarrow G'$ is of type ($A_m$) if it is of the form 
\[  g \mapsto \Phi c_g^m \delta_j(g) \, ,\]
where $\Phi$ is a $\Z$-linear endomorphism commuting with $c_g$, and  it is of type ($B_m$) if
\[ g \mapsto \Phi c_g^{m} b \,\]
with $b \in G'$ fixed. 

\medskip

In particular, if $F$ is of type $(A_m)$, we proved that 
\begin{equation} \label{eq: difference iterate}
\nabla_{r, h} F = a_{m} + b_{m} \, \end{equation}
where $a_m$ is of type $(A_{m})$ and $b_{m}$ is of type $(B_{m})$; and that when $r = m+1$, then $a_{m}=0$.

Now, we compute the effect of twisted differences on a map of type $(B_{m})$. Indeed,
\[ \nabla_{r, h} \Phi c_g^m b (g) = \Phi c_g^m c_h^m b- \Phi c_h^r c_g^m b = \Phi  (c_h^m - c_h^r) c_g^m b \, . \]
In particular, if $F$ is of type $(B_m)$, then $\nabla_{r, h} F$ is again of type $(B_m)$, and $\nabla_{m, h} F = 0$.

\smallskip

Therefore, a couple of convenient twisted differences suffice to vanish each summand of $W_a$, namely
$$\nabla_{m,\, h'} \nabla_{m+1,\,h} \Psi c_g^m \delta_j(g) = 0 \, . $$

\medskip

Now $ W_a$ is the sum $\sum_{j=1}^s a_{m_j}$ where each $a_m$ is of type $(A_m)$. Hence, for arbitrary $h_1, \dots, h_r \in G$ we get, by iterated applications of \eqref{eq: difference iterate}, 
\[\nabla_{m_r+1, \, h_r} \dots \nabla_{m_1+1, \, h_1} W_a 
= \sum_{j>r}\tilde{a}_{m_j}
+ \sum_{j=1}^s \tilde{b}_{m_j} \, ,
\]
where $\tilde{a}_{m_j}$ is of type $(A_{m_j})$, and each $\tilde{b}_{m_j}$ is of type $(B_{m_j})$. In particular, when $r = s$, we get 
\[\nabla_{m_s+1, \, h_s}\cdots \nabla_{m_1+1,h_1} W_a 
=  \sum_{j=1}^s \tilde{b}_{m_j} \, , \]
where each $\tilde{b}_{m_j}$ is of type $(B_{m_j})$.

Now, again by iterated applications of formula \eqref{eq: difference iterate}, and using that the twisted difference of a map of type $(B_m)$ is again of type $(B_{m})$,  we get
\[ \nabla_{m_s, \, h_{2s}}  \dots \nabla_{m_1, \, h_{s+1}} \nabla_{m_s+1,h_s} \dots \nabla_{m_1+1,h_1} W_a (g)= 0 \, ,\]
for any choice of elements $h_1, \dots, h_{2s} \in G$. 

\medskip

Finally, since the twisted difference of a central valued map is again central valued, in view of \eqref{eq: twisted central}, 
\[ \Delta_{h_{2s}}  \dots \Delta_{h_1} W_a (g) = \Delta_{h_{2s}}  \dots \Delta_{h_1} w (a_1, \dots, a_{i-1}, g, a_{i+1}, \dots, a_k)  = 0 \, ,
\]
for all $h_1, \dots, h_{2s}, g \in G$. Since the elements $a_j \in G$ are also arbitrary, we conclude that the word map $w$ is polynomial in the $i$-th variable. Since the variable was arbitrary, the word map $w$ is polynomial by \cite[Proposition 3.6]{Leibman}.
\end{proof}

The next step consists of writing an arbitrary commutator word $w \in M_k'$ as the product of reduced words. For that purpose, we use the Boolean cross-effect of $w$, see \cite[Proposition 2.4]{HPV}. More precisely, for each $S \subseteq \{ 1, \dots, k \}$, define 
\[ w_S = \sum_{T \subseteq S} { (-1)^{|S|- |T|}} \rho_T(w) \, . \]

This construction makes each $w_S$ a reduced word (only variables $i \in S$ appear on $w_S$).

\begin{lemma} \label{lem: reduced}
Let $w \in M_k$ be a $k$-variable commutator word and $S \subseteq \{1, \dots, k \}$. The word $w_S$ is reduced in each $i \in S$.
\end{lemma}
\begin{proof}
Since $\rho_T(w)$ is $M_k'$-valued for all $T \subseteq S$, we use additive notation, {\it i.e.}
\[ w_S = \sum_{T \subseteq S}  (-1)^{|S|- |T|} \rho_T(w) \, .\]

 Fix $i \in S$. For each $T \subseteq S \setminus \{i \}$, evaluating at $x_i = 1$, yields the same word in both $\rho_T(w)$ and $\rho_{T \cup \{i \}}(w)$, but they appear with opposite signs, and thus
\[ w_S(x_1, \dots, x_{i-1}, 1, x_{i+1}, \dots, x_k) = 0 \, . \qedhere \]
\end{proof}

From now on, we abbreviate $[k] = \{ 1 , \dots, k \}$.

\begin{lemma}
Let $G$ be a metabelian group, $w$ a $k$-variable commutator word. Then, 
\[w(\g) =\prod_{S\subseteq [k]}w_S (\g) \, ,
\]
for every $\g \in G^k$.
\end{lemma}

\begin{proof}
Using additive notation, we write 
\[ \sum_{S\subseteq [k]}w_S (\g) = \sum_{S \subseteq [k]} \sum_{T \subseteq S} (-1)^{|S|- |T|}\rho_T(w)(\g) \, . \]
By an inclusion-exclusion argument, all terms cancel out, except when $S=T= \{1, \dots, k \} $. Therefore,
\[ \sum_{S\subseteq [k]}w_S (\g) = \sum_{S \subseteq [k]} \sum_{T \subseteq S} (-1)^{|S|- |T|}\rho_T(w)(\g) = \rho_{\{1,\dots,k\}}(w)(\g)= w(\g)  \, . \qedhere \]
\end{proof}

\begin{proof}[Proof of Theorem~\ref{thm: polynomial}]
Since $w$ is $Z(G)$-valued, for each $S \subseteq \{1, \dots, k \}$, the cross-effect map $w_S$ is $Z(G)$-valued. Furthermore, by Lemma \ref{lem: reduced}, each $w_S$ is reduced, and thus, in view of Proposition~\ref{prop:  polynomial-reduced}, the word map $w_S \colon G^k \rightarrow Z(G)$ is polynomial. Finally, according to \cite[Theorem 1]{Hu}, the pointwise product $\prod_{S \subseteq [k]} w_S$ is polynomial as well. 
\end{proof}

\section{Proof of Theorem~\ref{thm: metabelian}} \label{sec: final}

We deduce the main theorem from the results of the previous two sections.

\begin{proof}[Proof of Theorem~\ref{thm: metabelian}]
Let $w$ be a group word and let $G$ be a profinite metabelian group. Suppose that $|w\{G \}| < 2^{\aleph_0}$. We shall prove that the set of $w$-values $w\{G\}$ is finite, and then the verbal subgroup $w(G)$ is finite by Turner-Smith ({\it loc. cit.}).

By Proposition~\ref{prop: commutator}, we assume that $w$ is a commutator word. In view of Proposition~\ref{prop: semiconcise}, $[w(G), G]$ is finite, so, passing to a finite quotient if needed, we assume that $w(G)$ is central in $G$. Hence, the word map $w \colon G^k \rightarrow w(G)$ is central valued, and, by Theorem~\ref{thm: polynomial}, it is a polynomial map. Since the image of $w$ has cardinality less than the continuum, it must be finite by Proposition~\ref{prop: polynomial strongly concise}. That is, $w\{G\}$ is finite.
\end{proof}

\bibliographystyle{plain} 
\bibliography{bibli} 

@article{A2,
 author = {Zozaya,~A.},
 title = {Conciseness in compact {{\(R\)}}-analytic groups},
 fjournal = {Journal of Algebra},
 journal = {J. Algebra},
 volume = {624},
 pages = {1--16},
 year = {2023}
}

@article{Detomi,
 author = {Detomi,~E.},
 title = {A note on strong conciseness in virtually nilpotent profinite groups},
 fjournal = {Archiv der Mathematik},
 journal = {Arch. Math. (Basel)},
 volume = {120},
 pages = {115--121},
 year = {2023}
}

@article{DKS,
 author = {Detomi,~E. and Klopsch,~B. and Shumyatsky,~P.},
 title = {Strong conciseness in profinite groups},
 fjournal = {Journal of the London Mathematical Society. Second Series},
 journal = {J. Lond. Math. Soc., II. Ser.},
 volume = {102},
 pages = {977--993},
 year = {2020}
}

@article{KS,
 author = {Khukhro,~E.~I. and Shumyatsky,~P.},
 title = {Strong conciseness of {Engel} words in profinite groups},
 fjournal = {Mathematische Nachrichten},
 journal = {Math. Nachr.},
 volume = {296},
 pages = {2404--2416},
 year = {2023}
}

@article{Iker,
 author = {Heras,~I.~de~las and Pintonello,~M. and Shumyatsky,~P.},
 title = {Strong conciseness of coprime commutators in profinite groups},
 fjournal = {Journal of Algebra},
 journal = {J. Algebra},
 volume = {633},
 pages = {1--19},
 year = {2023}
}

@article{DMS,
 author = {Detomi,~E. and Morigi,~M. and Shumyatsky,~P.},
 title = {Strong conciseness of coprime and anti-coprime commutators},
 fjournal = {Annali di Matematica Pura ed Applicata. Serie Quarta},
 journal = {Ann. Mat. Pura Appl.},
 volume = {200},
 pages = {945--952},
 year = {2021}
}

@article{Martina,
 author = {Conte,~M. and Petschick,~J.~M.},
 title = {Conciseness of first-order formulae},
 fjournal = {Monatshefte f{\"u}r Mathematik},
 journal = {Monatsh. Math.},
 volume = {209},
 pages = {215--240},
 year = {2026}
}

@article{A8,
 author = {Zozaya,~A.},
 title = {On generalisations of conciseness},
 fjournal = {Archiv der Mathematik},
 journal = {Arch. Math. (Basel)},
 volume = {125},
 pages = {599--603},
 year = {2025}
}

@book{Munkres,
 author = {Munkres,~J.~R.},
 title = {Topology},
 edition = {2nd},
 year = {2000},
 publisher = {Upper Saddle River},
 address = {Prentice Hall (NJ)}
}

@book{words,
 author = {Segal,~D.},
 title = {Words. {N}otes on verbal width in groups},
 fseries = {London Mathematical Society Lecture Note Series},
 series = {Lond. Math. Soc. Lect. Note Ser.},
 volume = {361},
 year = {2009},
 publisher = {Cambridge University Press},
 address = {Cambridge}
}

@article{TS,
 author = {Turner-Smith,~R.~F.},
 title = {Finiteness conditions for verbal subgroups},
 fjournal = {Journal of the London Mathematical Society},
 journal = {J. Lond. Math. Soc.},
 volume = {41},
 pages = {166--176},
 year = {1966}
}

@article{A7,
 author = {Heras,~I.~de~las and Zozaya,~A.},
 title = {Strong conciseness and equationally {N}oetherian groups},
 fjournal = {Annali di Matematica},
 journal = {Ann. Mat. Pura Appl.},
 year = {2026},
 volume={205},
 pages={749--758}
}

@article{Ba,
 author = {Bachmuth, S.},
 title = {Automorphisms of free metabelian groups},
 fjournal = {Transactions of the American Mathematical Society},
 journal = {Trans. Am. Math. Soc.},
 issn = {0002-9947},
 volume = {118},
 pages = {93--104},
 year = {1965}
}

@article{DMS2,
 author = {Detomi, E. and Morigi, M. and Shumyatsky, P.},
 title = {On countable coverings of word values in profinite groups.},
 fjournal = {Journal of Pure and Applied Algebra},
 journal = {J. Pure Appl. Algebra},
 issn = {0022-4049},
 volume = {219},
 pages = {1020--1030},
 year = {2015}
}

@article{merz,
 author = {Merzljakov,~Ju.~I.},
 title = {Verbal and marginal subgroups of linear groups},
 fjournal = {Soviet Mathematics. Doklady},
 journal = {Sov. Math., Dokl.},
 volume = {8},
 pages = {1538--1541},
 year = {1968}
}

@article{Ivanov,
 author = {Ivanov,~S.~V.},
 title = {P. {Hall}'s conjecture on the finiteness of verbal subgroups},
 fjournal = {Soviet Mathematics},
 journal = {Sov. Math.},
 volume = {33},
 number = {6},
 pages = {59--70},
 year = {1990}
}

@article{Her,
 author = {Herfort,~W.},
 title = {Factorizing profinite groups into two abelian subgroups.},
 fjournal = {International Journal of Group Theory},
 journal = {Int. J. Group Theory},
 issn = {2251--7650},
 volume = {2},
 number = {1},
 pages = {45--47},
 year = {2013}
}

@article{HPV,
 author = {Hartl, M. and Pirashvili, T. and Vespa, C.},
 title = {Polynomial functors from algebras over a set-operad and nonlinear {Mackey} functors},
 fjournal = {IMRN. International Mathematics Research Notices},
 journal = {Int. Math. Res. Not.},
 issn = {1073-7928},
 volume = {2015},
 number = {6},
 pages = {1461--1554},
 year = {2015}
}

@article{Hu,
 author = {Hu, Y.-Q.},
 title = {Polynomial maps and polynomial sequences in groups},
 fjournal = {Journal of Group Theory},
 journal = {J. Group Theory},
 issn = {1433-5883},
 volume = {27},
 number = {4},
 pages = {739--787},
 year = {2024}
}

@article{Passi,
 author = {Passi, I. B. S.},
 title = {Dimension subgroups},
 fjournal = {Journal of Algebra},
 journal = {J. Algebra},
 issn = {0021-8693},
 volume = {9},
 pages = {152--182},
 year = {1968}
}

@book{Tao,
 author = {Tao,~T.},
 title = {Higher order {Fourier} analysis},
 fseries = {Graduate Studies in Mathematics},
 series = {Grad. Stud. Math.},
 issn = {1065-7339},
 volume = {142},
 isbn = {978-0-8218-8986-2},
 year = {2012},
 publisher = {Providence, RI: American Mathematical Society (AMS)}
}

@article{bounded,
 author = {Fern{\'a}ndez-Alcober,~G.~A. and Morigi,~M.},
 title = {Outer commutator words are uniformly concise.},
 fjournal = {Journal of the London Mathematical Society. Second Series},
 journal = {J. Lond. Math. Soc., II. Ser.},
 issn = {0024-6107},
 volume = {82},
 number = {3},
 pages = {581--595},
 year = {2010}
}

@article{Ito,
 author = {It{\^o},~N.},
 title = {{\"U}ber das {Produkt} von zwei abelschen {Gruppen}},
 fjournal = {Mathematische Zeitschrift},
 journal = {Math. Z.},
 issn = {0025-5874},
 volume = {62},
 pages = {400--401},
 year = {1955}
}

@article{Schur,
 author = {Schur,~I.},
 title = {{\"U}ber die {Darstellung} der endlichen {Gruppen} durch gebrochene lineare {Substitutionen}.},
 fjournal = {Journal f{\"u}r die Reine und Angewandte Mathematik},
 journal = {J. Reine Angew. Math.},
 issn = {0075-4102},
 volume = {127},
 pages = {20--50},
 year = {1904}
}

@phdthesis{Stroud,
    author    = {Stroud,~P.~W.},
    title     = {Topics in the theory of verbal subgroups},
    school    = {University of Cambridge},
    year      = {1966},
    type      = {PhD thesis},
    address   = {Cambridge}
}

@article{Leibman,
 author = {Leibman,~A.},
 title = {Polynomial mappings of groups},
 fjournal = {Israel Journal of Mathematics},
 journal = {Isr. J. Math.},
 issn = {0021-2172},
 volume = {129},
 pages = {29--60},
 year = {2002}
}

@article{semiconciseness,
 author = {Delizia,~C. and Shumyatsky,~P. and Tortora,~A.},
 title = {On semiconcise words},
 fjournal = {Journal of Group Theory},
 journal = {J. Group Theory},
 issn = {1433-5883},
 volume = {23},
 number = {4},
 pages = {629--639},
 year = {2020}
}
\end{document}